\documentclass{amsart}
\usepackage{amssymb,amsmath,amstext,amsthm,amsfonts,amscd}
\usepackage{latexsym}
\usepackage{enumerate}
\usepackage{graphicx}
\newtheorem{theorem}{Theorem}
\newtheorem*{theorem*}{Theorem}
\newtheorem{proposition}[theorem]{Proposition}

\newtheorem{lemma}[theorem]{Lemma}

\newcommand{\Z}{\mathbb{Z}}
\newcommand{\Q}{\mathbb{Q}}
\newcommand{\R}{\mathbb{R}}
\newcommand{\C}{\mathbb{C}}
\newcommand{\Ci}{\mathbb{C}_\infty}

\newcommand{\SL}{\mathrm{SL}}
\newcommand{\GL}{\mathrm{GL}}
\newcommand{\PGL}{\mathrm{PGL}}
\newcommand{\re}{\mathrm{Re}}
\newcommand{\e}{\mathrm{e}}
\newcommand{\D}{\mathcal{D}}
\newcommand{\T}{\mathcal{T}}
\newcommand{\E}{\mathcal{E}}
\newcommand{\Po}{\mathcal{P}}
\newcommand{\den}{\mathrm{den}}
\newcommand{\aff}{\mathrm{aff}}
\newcommand{\Tr}{\mathrm{T}}

\newcommand{\M}{\mathcal{M}}
\newcommand{\Mat}{\mathrm{M}_{2\times 1}}
\newcommand{\Mer}{\mathrm{Mer}(\C;0)}
\newcommand{\Mirr}{\mathrm{M}^{\mathrm{Irr}}}
\newcommand{\sgn}{\mathrm{sgn}}
\newcommand{\diag}{\mathrm{diag}}
\newcommand{\Ber}{\bar{B}}
\newcommand{\Ze}{\mathcal{Z}}
\newcommand{\Zet}{\widetilde{\mathcal{Z}}}
\newcommand{\ind}{\mathbf{1}}

\newcommand{\fid}{\mathfrak{f}}
\newcommand{\bid}{\mathfrak{b}}
\newcommand{\cid}{\mathfrak{c}}
\newcommand{\aid}{\mathfrak{a}}
\newcommand{\gid}{\mathfrak{g}}
\newcommand{\Ok}{\mathcal{O}}
\newcommand{\G}{\mathrm{G}}
\newcommand{\W}{\mathrm{M}_{2\times1}^{\mathrm{Irr}}(\R)}
\newcommand{\ZZ}{\mathfrak{Z}}
\newcommand{\A}{\mathcal{A}}
\newcommand{\lcm}{\mathrm{lcm}}

\newcommand{\ie}{{\it{i.$\,$e.\ }}}

\begin{document}
\title[The Barnes-Hurwitz zeta cocycle at $s=0$]{The Barnes-Hurwitz zeta cocycle at $s=0$ and Ehrhart quasi-polynomials of triangles}
\author{Milton Espinoza}
\email{milton.espinozae@userena.cl}
\address{Departamento de Matem\'aticas, 
   Facultad de Ciencias, 
   Universidad de La Serena, 
   Juan Cisternas 1200, La Serena,
   Chile}
\thanks{This work was partially supported by the Chilean FONDECYT Iniciaci\'on grant 11200482}

\begin{abstract}
Following a theorem of David R. Hayes, we give a geometric interpretation of the special value at $s=0$ of certain $1$-cocycle on $\PGL_2(\Q)$ previously introduced by the author. This work yields three main results: an explicit formula for our cocycle at $s=0$, a generalization and a new proof of Hayes' theorem, and an elegant summation formula for the $0$th coefficient of the Ehrhart quasi-polynomial of certain triangles in $\R^2$. 
\end{abstract}

\maketitle

\section*{Introduction}
In \cite{Ha88}, Hayes gave the following geometric interpretation of the special value at $s=0$ of ray class zeta functions over real quadratic fields:
\begin{align}\label{Hayes.formula}
\zeta_{\fid}(\bid,0)-m(\fid)\zeta_{\Ok_k}(\bid\fid^{-1},0)=\Big(\sum_{\lambda\in 1+\Lambda}w_\fid(\lambda)\Big) - \frac{\mathrm{area}(\Delta_\fid)}{\mathrm{vol(\Lambda)}}-\frac{1}{2}.
\end{align}
On the left-hand side of the above formula we find arithmetic data: $\fid$ and $\bid$ are relatively prime integral ideals of a real quadratic field $k$ having ring of integers $\Ok_k$; $m(\fid)$ is the order of the quotient group $\Ok_{k,+}^\times/\Ok^\times_{k,+,\fid}$, where $\Ok_{k,+}^\times$ and $\Ok_{k,+,\fid}^\times$ are respectively the group of totally positive units of $k$ and the group of totally positive units of $k$ congruent to $1$ modulo $\fid$; $\zeta_\gid(\cid,0)$ is the value at $s=0$ of the meromorphic continuation to $\C$ of 
\begin{align*}
\zeta_\gid(\cid,s):=\sum_{\mathfrak{a}}\mathrm{N}(\mathfrak{a})^{-s}\qquad\qquad(\re(s)>1),
\end{align*}
where the sum runs over all integral ideals $\aid$ of $k$ lying in the same narrow ray class of $\cid$ modulo $\gid$, and $\mathrm{N}(\aid)=\#(\Ok_k/\aid)$ is the norm of $\aid$.

On the right-hand side of \eqref{Hayes.formula} we have geometric data attached to $k$. Indeed, fix a pair of embeddings $x\mapsto x^{(1)}$ and $x\mapsto x^{(2)}$ of $k$ into $\R$. Then identify $k$ with its image in $\R^2$ through $x\mapsto(x^{(1)},x^{(2)})$. Hence $\Lambda$ denotes $\bid^{-1}\fid$ as a lattice in $\R^2$, whereas $1+\Lambda$ is the affine lattice $1+\bid^{-1}\fid$; $\Delta_\fid$ is the triangle in $\R^2$ with vertices at $0$, $1$, and $\epsilon_\fid$, where $\epsilon_\fid$ is the generator of $\Ok_{k,+,\fid}^\times$ such that $\epsilon_\fid^{(1)}<\epsilon_\fid^{(2)}$; $\mathrm{area}(\Delta_\fid)$ is the usual area of $\Delta_\fid$;
\begin{align*}
\mathrm{vol}(\Lambda):=\left|\det\begin{pmatrix}
\rho^{(1)}&\rho^{(2)}\\ \tau^{(1)}&\tau^{(2)}
\end{pmatrix}\right|
\end{align*}
for any $\Z$-basis $\{\rho,\tau\}$ of $\bid^{-1}\fid$; $w_\fid$ is the real-valued function on $\R^2$ with support $\Delta_\fid$ such that $w_\fid$ equals $1$ on the interior of $\Delta_\fid$ and $1/2$ on its boundary.\footnote{Hayes' original version of \eqref{Hayes.formula} omits the $1/2$ on the right-hand side at the cost of shrinking the support of $w_\fid$ by discarding a portion of the boundary of $\Delta_\fid$.}

The proof of \eqref{Hayes.formula} uses Shintani's method \cite{Sh76}, which in particular yields the formula 
\begin{align}\label{Shintani.formula.s=0}
\zeta_\gid(\cid,0)=\frac{1}{2}\left(\alpha^{(1)}+\alpha^{(2)}\right),
\end{align}
where $\alpha\in k$ is an algebraic number depending on $\gid$ and $\cid$. The most challenging part of Hayes' proof comes from comparing $\zeta_{\fid}(\bid,0)$ with $m(\fid)\zeta_{\Ok_k}(\bid\fid^{-1},0)$, which required him to cleverly decompose a certain cone-shaped fundamental domain for the multiplicative action of $\Ok_{k,+,\fid}$ on $\R_{>0}^2$, as well as to choose a very special common coordinate system to make computations. In \cite[p. 212]{Ha88}, Hayes states that it would be interesting to derive \eqref{Hayes.formula} directly from Siegel's formula for $\zeta_\gid(\cid,0)$, which essentially asks for obtaining \eqref{Hayes.formula} through manipulations of certain Bernoulli polynomials. We cover somewhat the latter approach in this paper.

Instead of using \eqref{Shintani.formula.s=0} as our model, we follow a result due to Yoshida \cite[Ch. IV, Cor 6.3]{Yo03}, which implies that the $\alpha$ in \eqref{Shintani.formula.s=0} is actually rational, so 
\begin{align}\label{Yoshida.formula.s=0}
\zeta_\gid(\cid,0)=\alpha.
\end{align}
Hence the study of $\zeta_\gid(\cid,0)$ is equivalent to the study of certain linear combinations of the Barnes' double zeta function $\zeta_2(s,\omega,s)$ and the Hurwitz zeta function $\zeta(z,s)$ at $s=0$, where $\zeta_2(s,\omega,s)$ is defined as the meromorphic continuation to $\C$ of the absolutely convergent Dirichlet series
\begin{align*}
\zeta_2(z,\omega,s):=\sum_{m=0}^\infty\sum_{n=0}^\infty(z+m\omega_1+n\omega_2)^{-s}, \qquad\omega=\begin{pmatrix}
\omega_1\\\omega_2
\end{pmatrix},
\end{align*}
for $z, \omega_1, \omega_2>0$ and $\re(s)>2$, and
$\zeta(z,s)$ is the meromorphic continuation to $\C$ of the absolutely convergent Dirichlet series
\begin{align*}
\zeta(z,s):=\sum_{n=0}^\infty(z+n)^{-s} \qquad(z>0, \ \re(s)>1).
\end{align*}

In \cite{Es22} we studied a 1-cocycle $\ZZ$ on $\PGL_2(\Q)$ that organizes algebraically the aforementioned linear combinations (for any complex number $s$), and the motivation of this paper is to reproduce Equations \eqref{Yoshida.formula.s=0} and \eqref{Hayes.formula} from this standpoint. In fact, we generalize both equations and we also use different techniques in order to get them (see Theorems \ref{ZZ_0.thm} and \ref{Hayes.Espinoza.thm.}). Precisely, whereas Yoshida's proof of \eqref{Yoshida.formula.s=0} makes use of Baker's theorem on logarithms of algebraic numbers,\footnote{It should be mentioned that Yoshida's result includes totally real number fields in general.} ours uses Fourier analysis on the real line. On the other hand, in addressing \eqref{Hayes.formula}, we replace Hayes' decomposition of the conical fundamental domain above with the use of elementary properties of the periodic Bernoulli functions
\begin{align}\label{Bernoulli.defi}
\Ber_2(y):=\{y\}^2-\{y\}+1/6 \qquad\text{and}\qquad \Ber_1(y):=\begin{cases}0&\text{if $y\in\Z$,}\\ \{y\}-1/2& \text{if $y\notin\Z$,}\end{cases}
\end{align}
where $\{y\}:=y-\lfloor y\rfloor$ for any $y\in\R$. 

Finally, as an application of our generalized version of Equation \eqref{Hayes.formula} and its interpretation in terms of a group cocycle, we give an elegant summation formula involving a part of the lattice point enumerator 
\begin{align*}
\G(T,m):=\#(\Z^2\cap m T) \qquad\qquad (m\in\Z_{\geq0})
\end{align*}
of certain triangles $T\subset\R^2$ having vertices in $\Q^2$ and parameterized by $\SL_2(\Z)$ matrices, thus making contact with the Ehrhart theory of polytopes. More precisely, it is known that $\G(T,m)$ is a \textit{quasi-polynomial in $m$} \cite{Eh62}, \ie
\begin{align*}
\G(T,m)=\G_2(T,m)m^2+\G_1(T,m)m+\G_0(T,m) \qquad(m\in\Z_{\geq0}),
\end{align*}
where each $\G_i(T):\Z_{\geq0}\to\R$ is a function satisfying $\G_i(T,m+d)=\G_i(T,m)$ for all $m\in\Z_{\geq0}$, where $d$ is the smallest positive integer such that $dT$ has vertices in $\Z^2$. For the triangles $T$ we are interested in, a theorem of McMullen \cite{Mc78} implies that $\G_1(T)$ and $\G_2(T)$ have actually period $1$, so they are constants. Moreover, $\G_1(T,m)$ and $\G_2(T,m)$ can be explicitly computed in terms of $d$, $\mathrm{area}(dT)$, and the number of integral points in the boundary of $dT$. Hence the computation of $\G(T,m)$ is reduced to the study of $\G_0(T,m)$, the latter being the main ingredient of our summation formula \eqref{Ehrhart.Espinoza.formula}. It is worth mentioning that we have not been able to find any parallel of \eqref{Ehrhart.Espinoza.formula} in the literature.

This article is organized as follows. In Section \ref{Main.results.sect.} we introduce the bulk of the notation employed subsequently, and we also state our main results without proofs. Then we address respectively the proofs of Theorems \ref{ZZ_0.thm}, \ref{Hayes.Espinoza.thm.}, and \ref{Ehrhart.Espinoza.thm} in Sections \ref{ZZ_0.sect.}, \ref{Hayes.Espinoza.sect.}, and \ref{Ehrhart.Espinoza.sect.}.

\section{Main results}\label{Main.results.sect.}
From now on $\M$ denotes the set of $2\times2$ integer matrices having nonzero determinant, $\W$ is the set of $2\times1$ real matrices having linearly independent entries over $\Q$, and $G:=\PGL_2(\Q)$. 

Let $\A_0$ be the set of all continuous functions $F:\W\to\R$ such that $F(\alpha\omega)=F(\omega)$ for all $\alpha\in\R_{>0}$ and $\omega\in\W$, where $\W$ and $\R$ have the Euclidean topology. Then $\A_0$ is an additive abelian group and there is an action `$\star$' of the multiplicative monoid $\M$ on $\A_0$ given by 
\begin{align*}
(\gamma\star F)(\omega):=\sgn(\det\gamma)F(\gamma^\Tr\omega) \qquad\text{for all $\gamma\in\M$, $F\in\A_0$, and $\omega\in\W$.}
\end{align*}
Let $D_0=D_0(\Q^2/\Z^2,\A_0)$ denote the set of \textit{$\A_0$-valued homogeneous distributions}, \ie the collection of all maps $\nu:\Q^2/\Z^2\to\A_0$ such that 
\begin{align}\label{dist.rela}
\nu(x)=\sum_{\substack{y\in\Q^2/\Z^2\\cy=x}}\nu(y)\qquad\text{for all $x\in\Q^2/\Z^2$ and $c\in\Z_{>0}$.}
\end{align}
Then $D_0$ is an additive abelian group acted on by $\M$ via
\begin{align*}
(\gamma\cdot\nu)(x):=\gamma\star\sum_{\substack{y\in\Q^2/\Z^2\\y\gamma^\Tr=x}}\nu(y) \qquad\text{for all $\gamma\in\M$, $\nu\in D_0$, and $x\in\Q^2/\Z^2$.}
\end{align*} 
Hence we use \eqref{dist.rela} to extend `$\cdot$' to an action of $\GL_2(\Q)$ on $D_0$. Indeed, for any $\gamma\in\GL_2(\Q)$ and any $\nu\in D_0$, we define $\gamma\cdot \nu:=(c\gamma)\cdot\nu$ for any $c\in\Z_{>0}$ such that $c\gamma\in\M$. Since $\nu$ satisfies \eqref{dist.rela}, the above extension is well-defined (\ie it is independent of the choice of $c$). Finally, it follows that $G$ acts on the $\GL_2(\Q)$-submodule $\D_0:=D_0^{\langle -I_2\rangle}$ of $D_0$ fixed by $-I_2$. Note that `$\cdot$' can be alternatively given by the identity
\begin{align*}
(\gamma\cdot \nu)(x,\omega)=\sgn(\det\gamma)\sum_{\mu\in\Z^2/\Z^2\gamma^\Tr}\nu\big((x+\mu)\gamma^{-\Tr},\gamma^\Tr\omega\big)
\end{align*}
valid for all $\gamma\in\M$, $\nu\in D_0$, and $(x,\omega)\in\Q^2/\Z^2\times\W$.

In order to state our generalization of \eqref{Yoshida.formula.s=0}, it only remains to define what will turn out to be an explicit version of $\ZZ$ at $s=0$ (see \cite[Thm 1]{Es22}). Firstly, recall the Dedekind-Rademacher sum, which is defined via the formula
\begin{align}\label{Dede-Rade.sum}
S(a,c;x)=\sum_{m\in\Z/c\Z}\Ber_1\!\Big(x_1+(x_2+m)\frac{a}{c}\Big)\Ber_1\!\Big((x_2+m)\frac{1}{c}\Big)
\end{align}
valid for any $a, c\in\Z$ with $c\not=0$, and any $x\in\R^2$. Then, for any
\begin{align*}
\gamma=\begin{pmatrix}a&b\\c&d\end{pmatrix}\in\M \qquad\text{and}\qquad (x,\omega)\in\Q^2\times\W,
\end{align*}
put $\ZZ_0(\gamma)(x,\omega):=0$ if $c=0$ and 
\begin{align}\label{ZZ_0.defi.}
\ZZ_0(\gamma)(x,\omega) \, := \, \frac{\gcd(a,c)^2}{2c(a+c\tau)}\,\Ber_2\!\Big(\frac{c}{\gcd(a,c)}x_1-\frac{a}{\gcd(a,c)}x_2\Big)\ + \ \frac{a+c\tau}{2c}\,\Ber_2(x_2)\\\nonumber + \ \sgn(c)\,S(-a,c;x) \ - \ \frac{\sgn(c(a+c\tau))}{4}\,\ind_{\Z}(x_1)\,\ind_{\Z}(x_2)
\end{align}
otherwise, where $\tau:=\omega_2/\omega_1$ and $\ind_X$ is the indicator function of a subset $X$ of $\R$.

\begin{theorem}\label{ZZ_0.thm}
The function $\ZZ_0:G\to\D_0$ induced by the rule above is an inhomogeneous 1-cocycle whose class in the first cohomology group $H^1(G,\D_0)$ is nonzero. Furthermore, let
\begin{align*}
\Ze_0:=\ZZ_0\!\begin{pmatrix}0&-1\\1&0\end{pmatrix}, \quad \gamma=\begin{pmatrix}a&b\\c&d\end{pmatrix}\in\M, \quad\text{and}\quad (x,\omega)\in\Q^2\times\W.
\end{align*}
\begin{enumerate}[$(i)$]
\item If both $\omega_1$ and $\omega_2$ are positive, then 
\begin{align*}
\Ze_0(x,\omega)=\zeta_2\big(z(x,\omega),\omega,0\big)+\frac{\ind_\Z(x_1)}{2}\zeta\big(\langle x_2\rangle,0\big)+\frac{\ind_\Z(x_2)}{2}\zeta\big(\langle x_1\rangle,0\big),
\end{align*}
where $\langle\lambda\rangle$ is the unique real number in the interval $(0,1]$ such that $\langle\lambda\rangle-\lambda$ is an integer for any $\lambda\in\R$ and $z(x,\omega):=\langle x_1\rangle\omega_1+\langle x_2\rangle\omega_2$.
\item In general,
\begin{align*}
\ZZ_0(\gamma)=\begin{pmatrix}1&a\\0&c\end{pmatrix}\cdot\Ze_0 \qquad\text{if $c\not=0$.}
\end{align*} 
\item\label{ZZ_0.thm.rati.} Suppose that $\gamma\in\SL_2(\Z)$ and there exists $\alpha\in\R$ such that 
\begin{align*}
\gamma^\Tr\omega=\alpha\omega\qquad\text{and}\qquad\alpha^{-1}v\equiv v\mod\Lambda,
\end{align*}
where $v:=x_1\omega_1+x_2\omega_2$ and $\Lambda:=\Z\omega_1\oplus\Z\omega_2$. Then 
\begin{align*}
\ZZ_0(\gamma)(x,\omega)=\frac{a+d}{2c}\Ber_2(x_2)+\sgn(c)S(-a,c;x)-\frac{\sgn(c\alpha)}{4}\ind_\Z(x_1)\ind_\Z(x_2),
\end{align*}
so $\ZZ_0(\gamma)(x,\omega)$ is a rational number.
\end{enumerate}
\end{theorem}

By using Shintani's method it can be shown that the above theorem gives $\zeta_\gid(\cid,0)=\ZZ_0(\gamma)(x,\omega)$ for certain choices of $\gamma$ and $(x,\omega)$ depending on $\gid$ and $\cid$, where $\gamma$ and $(x,\omega)$ satisfy the hypothesis of Theorem \ref{ZZ_0.thm}\eqref{ZZ_0.thm.rati.}.

Now we state our generalization of Equation \eqref{Hayes.formula}. For any
\begin{align*}
\gamma=\begin{pmatrix}a&b\\c&d\end{pmatrix}\in\SL_2(\Z)
\end{align*}
define \textit{the content of $\gamma$} as the positive integer $\ell_\gamma:=\gcd(c,a-1)$. Also, if $c\not=0$, denote by $T_{\gamma}$ the triangle in $\R^2$ with vertices at $(0,0)$, $(a-1,c)$, and $(-1,0)$. Then for any $x=(x_1,x_2)\in\Q^2$ with $x_1\notin\Z$ define $T_{\gamma,x}$ as the dilated triangle $\{x_1\}T_\gamma$.

\begin{theorem}\label{Hayes.Espinoza.thm.}
With the notation of the previous paragraph, let $\gamma\in\SL_2(\Z)$ with $c>0$ and $\ell_\gamma>1$, let $x=(x_1,0)\in\Q^2$ with $x_1\in\ell_\gamma^{-1}\Z\smallsetminus\Z$, and let $\omega\in\Mirr_{2\times1}(\R)$. Then
\begin{align}\label{Hayes.Espinoza.formula}
\ZZ_0(\gamma)(x,\omega)-\ZZ_0(\gamma)(\ell_\gamma x,\ell_\gamma^{-1}\omega)=
\Big(\sum_{\lambda\in\Z^2}w_{\gamma,x}(\lambda)\Big)-\mathrm{area}(T_{\gamma,x})+\frac{\sgn(a+c\tau)-3}{4},
\end{align}
where $\tau:=\omega_2/\omega_1$, and $w_{\gamma,x}$ is the real-valued function on $\R^2$ with support $T_{\gamma,x}$ such that $w_{\gamma,x}$ equals $1$ on the interior of $T_{\gamma,x}$ and $1/2$ on its boundary.
\end{theorem}

Hayes' formula \eqref{Hayes.formula} follows essentially from the above theorem by using Shintani's method except for an evident small change. Indeed, we have chosen to present the right-hand side of \eqref{Hayes.Espinoza.formula} in the canonical basis of $\R^2$ because it is then obvious the way in which $\omega$ enters the equation. We remark that Proposition~\ref{Ehrhart.Espinoza.Prop} below extends Equation~\eqref{Hayes.Espinoza.formula} by using the Ehrhart coefficient $\G_0$.

Now we state our last theorem assuming the notation given before Theorem \ref{Hayes.Espinoza.thm.}.

\begin{theorem}\label{Ehrhart.Espinoza.thm}
Let
\begin{align*}
\gamma=\begin{pmatrix}a&b\\c&d\end{pmatrix}, \ \gamma'=\begin{pmatrix}a'&b'\\c'&d'\end{pmatrix} \ \in \ \SL_2(\Z).
\end{align*}
Set $\ell:=\gcd(\ell_{\gamma},\ell_{\gamma'})$. If $c$, $c'$, and $ca'+dc'$ are positive, then 
\begin{align}\label{Ehrhart.Espinoza.formula}
\G_0(\ell^{-1}T_{\gamma\gamma'}\,,\,m)\,=\,\G_0(\ell^{-1}T_\gamma\,,\,m)\,+\,\G_0(\ell^{-1}T_{\gamma'}\,,\,m)\,+\,\{m/\ell\}\,-\,1
\end{align}
for all $m\in\Z_{\geq0}$.
\end{theorem}

\section{Proof of Theorem \ref{ZZ_0.thm}}\label{ZZ_0.sect.}
The proof of Theorem \ref{ZZ_0.thm} relies upon two main ideas. Firstly, we consider the Barnes-Hurwitz zeta cocycle $\ZZ$ introduced in \cite[Thm 1]{Es22}, and we remark that its value $\ZZ_0$ at $s=0$ preserves the cocycle property. Secondly, in order to obtain the explicit formula \eqref{ZZ_0.defi.} for $\ZZ_0$, we use the Fourier expansion of the periodic Bernoulli function $\Ber_2$ defined in \eqref{Bernoulli.defi}.

We first summarize the construction of $\ZZ$ for the sake of completeness. For any $\epsilon>0$, let $C(\epsilon)$ be the Hankel contour in the complex plane, \ie the counterclockwise oriented path consisting of the interval $[\epsilon,+\infty)$, and the circle of radius $\epsilon$ centered at the origin followed by the same interval. We define the $\Ci$-valued function $\Zet$ on the disjoint union 
\begin{align*}
(\Q^2\smallsetminus\Z^2)\times\mathrm{M}_{2\times1}(\R^\times)\times\C \quad \bigcup \quad \Z^2\times\mathrm{M}_{2\times1}(\R_{>0})\times\C
\end{align*}
by the rule
\begin{align}\label{Z.inco.}
&\Zet(x,\omega,s):=\mathcal{L}_2(x,\omega,s)+\frac{1}{2}\mathbf{1}_\Z(x_1)\mathcal{L}_1(x_2,\omega_2,s)+\frac{1}{2}\mathbf{1}_\Z(x_2)\mathcal{L}_1(x_1,\omega_1,s),
\end{align}
where we write $x=(x_1,x_2)$, $\omega=(\omega_1,\omega_2)^\Tr$,
\begin{align}
&\label{L2}\mathcal{L}_2(x,\omega,s):=\frac{1}{\Gamma(s)(\e(s)-1)}\int_{C(\epsilon)}\frac{\e^{-u(\langle x_1\rangle\omega_1+\langle x_2\rangle\omega_2)}}{(1-\e^{-u\omega_1})(1-\e^{-u\omega_2})}u^{s-1}du, \quad \text{and}\\
&\label{L1}\mathcal{L}_1(x_j,\omega_j,s):=\frac{1}{\Gamma(s)(\e(s)-1)}\int_{C(\epsilon)}\frac{\e^{-u\langle x_j\rangle\omega_j}}{1-\e^{-u\omega_j}}u^{s-1}du\qquad (j=1,2).
\end{align}
Here, $\Ci$ denotes the extended complex plane, $\e(s):=\e^{2\pi is}$ for all $s\in\C$, $\epsilon=\epsilon(\omega)>0$ is such that $u=0$ is the only zero of $(1-\e^{-u\omega_1})(1-\e^{-u\omega_2})$ in the closed disk of radius $\epsilon$ centered at the origin, and $u^{s-1}$ either considers $\arg(u)=0$ if $u$ lies in the negatively oriented interval $[\epsilon,+\infty)$ or $0<\arg(u)\leq 2\pi$ otherwise. The functions $\mathcal{L}_2$ and $\mathcal{L}_1$ converge and are independent of the choice of the collection $\{\epsilon(\omega)\}_\omega$. Thus they define meromorphic functions of $s$ on $\C$ for each suitable $(x,\omega)$, and one shows that
\begin{align}\label{L2.Barnes.Eq.}
\mathcal{L}_2(x,\omega,s)=\zeta_2\big(z(x,\omega),\omega,s\big), \qquad z(x,\omega):=\langle x_1\rangle\omega_1+\langle x_2\rangle\omega_2,
\end{align}
for all $(x,\omega,s)\in \Q^2\times\mathrm{M}_{2\times1}(\R_{>0})\times\C$, whereas \begin{align*}
\mathcal{L}_1(x_j,\omega_j,s)=\omega_j^{-s}\zeta(\langle x_j\rangle,s), \qquad j\in\{1,2\},
\end{align*}
for all $(x_j,\omega_j,s)\in\Q\times\R_{>0}\times\C$. Since the integrals in Equations \eqref{L2} and \eqref{L1} do not simultaneously converge for all $s\in\C$ when we assume both $x\in\Z^2$ and $\omega\notin\Mat(\R_{>0})$, we need to somehow extend the definition of $\Zet$. For this reason, we define the $\Ci$-valued function $\Ze$ on $\Q^2\times\mathrm{M}_{2\times1}(\R^\times)\times\C$ by
\begin{align*}
\Ze(x,\omega,s):=\det(B_\omega)\Zet(xB_\omega,B_\omega\omega,s),
\end{align*}
where $B_\omega$ is the diagonal matrix $\diag(\sgn(\omega_1),\sgn(\omega_2))$ for $\omega=(\omega_1,\omega_2)^\Tr$. Then \cite[Prop. 18]{Es22} shows that $\Ze:\Q^2/\Z^2\to\A$ is an \textit{$s$-distribution} \cite[Def. 4]{Es22}, where $\A$ is the collection of certain maps of the form $\Mirr_{2\times1}(\R)\to\Mer$ and $\Mer$ is the set of $\Ci$-valued meromorphic functions on $\C$ holomorphic at $s=0$. In the sequel, all we need to know about these $s$-distributions is that they form a $G$-module $\D$, and that ``evaluation at $s=0$'' induces a $G$-module homomorphism $h:\D\to\D_0$ thanks to \cite[Prop. 9]{Es22}. The $G$-action `$\cdot$' on $\D$ (by abuse of notation) is given by the identity
\begin{align*}
(\gamma\cdot \nu)(x,\omega,s)=\sgn(\det\gamma)\sum_{\mu\in\Z^2/\Z^2\gamma^\Tr}\nu\big((x+\mu)\gamma^{-\Tr},\gamma^\Tr\omega,s\big)
\end{align*}
in $\Ci$, valid for all $\gamma\in\M$, $\nu\in \D$, and $(x,\omega,s)\in\Q^2/\Z^2\times\W\times \C$. Hence the inhomogeneous $1$-cocycle $\ZZ:G\to\D$ is defined by 
\begin{align}\label{ZZ.defi.}
\mathfrak{Z}\!\begin{pmatrix}
a&b\\c&d
\end{pmatrix}=\begin{cases}\begin{pmatrix}1&a\\0&c\end{pmatrix}\cdot\Ze & \text{if $c\not=0$,}\\ 0 & \text{if $c=0$.} \end{cases}
\end{align}
Therefore the composition $h\circ \ZZ$ is an inhomogeneous $1$-cocycle of $G$ with values in $\D_0$. We would like to prove that $\ZZ_0=h\circ \ZZ$ (see \eqref{ZZ_0.defi.}).

The residue theorem allows us to compute $\Ze_0(x,\omega):=\Ze(x,\omega,0)$. Indeed, 
\begin{align}\label{Z.at.0}
\Ze_0(x,\omega)=\frac{1}{2}\frac{\omega_1}{\omega_2}\Ber_2(x_1)+\frac{1}{2}\frac{\omega_2}{\omega_1}\Ber_2(x_2)+\Ber_1(x_1)\Ber_1(x_2)-\frac{1}{4}\sgn\Big(\frac{\omega_2}{\omega_1}\Big)\ind_\Z(x_1)\ind_\Z(x_2)
\end{align}
for all $(x,\omega)\in\Q^2\times\mathrm{M}_{2\times1}(\R^\times)$. So in order to get $\ZZ_0$ explicitly, we must first learn to sum certain values of $\Ber_2$. For this purpose, we will employ the absolutely convergent Fourier expansion 
\begin{align}\label{B_2.Fourier}
\Ber_2(y)=\frac{1}{2\pi^2}\sum_{m\not=0}\frac{\e(my)}{m^2} \qquad\text{(valid for all $y\in\R$)}
\end{align}
together with the following lemma. We recall that $\e(z):=\e^{2\pi iz}$ for all $z\in\C$.

\begin{lemma}\label{character.order.lemma}
Let $\gamma\in\M$. For any $j\in\{1,2\}$ consider the group homomorphism $\chi_j:\Z^2/\Z^2\gamma^\Tr\longrightarrow\C^\times$ induced by
\begin{align}\label{character.def.}
\chi_j(\mu)=\e(\langle \mu, \gamma_{j*}^{-1}\rangle)
\qquad\text{for all $\mu\in\Z^2$,}
\end{align}
where $\gamma_{j*}^{-1}$ denotes the $j$th row of $\gamma^{-1}$ and $\langle *, *\rangle$ denotes the usual dot product on $\R^2$. Then if $$\gamma=\begin{pmatrix}a&b\\c&d\end{pmatrix},$$ the order of $\chi_j$ in the character group of $\Z^2/\Z^2\gamma^\Tr$ is 
\begin{align*}
\frac{\det \gamma}{\gcd(b,d)} \quad \text{if $j=1$} \qquad\text{and}\qquad \frac{\det \gamma}{\gcd(a,c)} \quad \text{if $j=2$}.
\end{align*}
\end{lemma}

\begin{proof}
Let $m$ be a nonzero integer. Note that $\chi_1^m=1$ if and only if 
\begin{align*}
m\langle \mu, \gamma_{1*}^{-1}\rangle=(\det \gamma)^{-1}(\mu_1md-\mu_2mb) \ \in \ \Z \qquad\text{for all $\mu\in\Z^2$.}
\end{align*}
Then $\chi_1^m=1$ if and only if $md$ and $mb$ are multiples of $\det \gamma$, which is equivalent to the fact that both $(\det \gamma)/\gcd(b,\det \gamma)$ and $(\det \gamma)/\gcd(d,\det \gamma)$ divide $m$. Similarly we obtain that $\chi_2^m=1$ if and only if both $(\det \gamma)/\gcd(a,\det \gamma)$ and $(\det \gamma)/\gcd(c,\det \gamma)$ divide $m$. Therefore
\begin{align}\label{character.order.cond.}
\chi_j^m=1 \quad\Longleftrightarrow\quad m\in\begin{cases}\lcm\Big(\frac{\det \gamma}{\gcd(b,\det \gamma)},\frac{\det \gamma}{\gcd(d,\det \gamma)}\Big)\,\Z & \text{if $j=1$,}\\ \lcm\Big(\frac{\det \gamma}{\gcd(a,\det \gamma)},\frac{\det \gamma}{\gcd(c,\det \gamma)}\Big)\,\Z & \text{if $j=2$.}\end{cases}
\end{align}
On the other hand, using elementary properties of the greatest common divisor one shows that 
\begin{align*}
\frac{|x|\gcd(x,y,z)}{\gcd(x,y)\gcd(x,z)}=\gcd\!\Big(\frac{x}{\gcd(x,z)},\frac{x}{\gcd(x,y)}\Big) \qquad\text{for all $x, y, z\in\Z$ with $x\not=0$.}
\end{align*}
The last formula together with the relation $\lcm(x,y)\gcd(x,y)=|xy|$ allows us to rewrite \eqref{character.order.cond.} as
\begin{align*}
\chi_j^m=1 \quad\Longleftrightarrow\quad m\in\begin{cases}\frac{|\det \gamma|}{\gcd(b,d,\det \gamma)}\,\Z & \text{if $j=1$,}\\ \frac{|\det \gamma|}{\gcd(a,c,\det \gamma)}\,\Z & \text{if $j=2$.}\end{cases}
\end{align*}
Finally the lemma follows from the remark that $\gcd(b,d,\det \gamma)=\gcd(b,d)$ and $\gcd(a,c,\det \gamma)=\gcd(a,c)$.
\end{proof}

\begin{lemma}\label{B_2.sum}
With the notation of the previous lemma, let $$\gamma=\begin{pmatrix}
a&b\\c&d
\end{pmatrix}\ \in \ \M$$ and let $R_\gamma$ be a complete set of coset representatives for $\Z^2$ modulo $\Z^2\gamma^\Tr$. Then, for any $j\in\{1,2\}$ and any $x\in\R^2$ we have that
\begin{align*}
\sum_{\mu\in R_\gamma}\Ber_2\big(\langle x+\mu,\gamma_{j*}^{-1}\rangle\big)=\begin{cases}\frac{\gcd(b,d)^2}{|\det \gamma|}\Ber_2\!\Big(\frac{\det \gamma}{\gcd(b,d)}\langle x,\gamma_{1*}^{-1}\rangle\Big) \quad\text{if $j=1$,}\\
\frac{\gcd(a,c)^2}{|\det \gamma|}\Ber_2\!\Big(\frac{\det \gamma}{\gcd(a,c)}\langle x,\gamma_{2*}^{-1}\rangle\Big)\quad\text{if $j=2$.}\end{cases}
\end{align*}
In particular, the above sum is independent of the choice of $R_\gamma$ and it remains invariant under the substitution $\gamma\mapsto\lambda \gamma$ for any nonzero integer $\lambda$.
\end{lemma}

\begin{proof}
We use \eqref{B_2.Fourier} and \eqref{character.def.} to show that
\begin{align*}
\sum_{\mu\in R_\gamma}\Ber_2\big(\langle x+\mu,\gamma_{1*}^{-1}\rangle\big)=\frac{1}{2\pi^2}\sum_{m\not=0}\frac{\e\big(m\langle x,\gamma_{1*}^{-1}\rangle\big)}{m^2}\sum_{\mu\in\Z^2/\Z^2\gamma^\Tr}\chi_1^m(\mu).
\end{align*}
Then we use Schur orthogonality relations for $\chi_1$, Lemma~\ref{character.order.lemma}, and \eqref{B_2.Fourier} to show that the right-hand side of the last equation equals 
\begin{align*}
\frac{\gcd(b,d)^2}{|\det \gamma|}\,\Ber_2\!\Big(\frac{\det \gamma}{\gcd(b,d)}\langle x,\gamma_{1*}^{-1}\rangle\Big).
\end{align*}
A similar reasoning yields the desired formula when $j=2$.
\end{proof}

To treat the more delicate addition concerning $\Ber_1$ we use the Dedekind-Rademacher sum \eqref{Dede-Rade.sum}. So it only remains to deal with certain sums of values of the form 
\begin{align*}
C(x,\omega):=\frac{1}{4}\sgn(\omega_2/\omega_1)\ind_\Z(x_1)\ind_\Z(x_2) \qquad\qquad\big(x\in\Q^2, \,\omega\in\mathrm{M}_{2\times1}(\R^\times)\big)
\end{align*} 
as in \eqref{Z.at.0}. In fact, we would like to compute
\begin{align*}
\sgn(c)\sum_{\mu\in\Z^2/\Z^2\sigma^\Tr}C\big((x+\mu)\sigma^{-\Tr},\sigma^\Tr\omega\big) \qquad\qquad\text{for $\sigma=\begin{pmatrix}
1&a\\0&c
\end{pmatrix}$ with $c\not=0$,}
\end{align*}
which is actually
\begin{align*}
\frac{\sgn(c(a+c\tau))}{4}\sum_{\mu\in\Z^2/\Z^2\sigma^\Tr}\ind_\Z\Big((x_1+\mu_1)-(x_2+\mu_2)\frac{a}{c}\Big)\,\ind_\Z\Big((x_2+\mu_2)\frac{1}{c}\Big),
\end{align*}
where $\tau:=\omega_2/\omega_1$. Note that each addend equals zero unless both arguments of the indicator function $\ind_\Z$ are integers. An elementary analysis shows that this occurs if and only if $x\in\Z^2$ and $\mu_2\equiv-x_2\mod c$. Furthermore, for a fixed $x\in\Z^2$, there exists a unique pair $\mu\in\Z^2$, modulo $\Z^2\sigma^\Tr$, such that $\mu_2\equiv-x_2\mod c$. Then
\begin{align}\label{Char.sum.}
\sgn(c)\sum_{\mu\in\Z^2/\Z^2\sigma^\Tr}C\big((x+\mu)\sigma^{-\Tr},\sigma^\Tr\omega\big)=\frac{\sgn(c(a+c\tau))}{4}\ind_\Z(x_1)\ind_\Z(x_2).
\end{align}
Now we address the proof of Theorem~\ref{ZZ_0.thm}.

\begin{proof}[Proof of Theorem~\ref{ZZ_0.thm}]
For $c\not=0$ let $$\gamma=\begin{pmatrix}
a&b\\c&d
\end{pmatrix} \qquad\text{and}\qquad\sigma=\begin{pmatrix} 
1&a\\0&c
\end{pmatrix}\qquad\text{in $\M$}.$$ 
Using Equations \eqref{ZZ.defi.}, \eqref{Z.at.0}, and \eqref{Char.sum.}, one shows that
\begin{align*}
&(h\circ\ZZ)(\gamma)(x,\omega)=\frac{\sgn(c)}{2}\sum_{\mu\in\Z^2/\Z^2\sigma^\Tr}\Big(\frac{\omega_1}{a\omega_1+c\omega_2}\Ber_2(\langle x+\mu,\sigma^{-1}_{1*}\rangle)\\&+\frac{a\omega_1+c\omega_2}{\omega_1}\Ber_2(\langle x+\mu,\sigma^{-1}_{2*}\rangle)
+2\Ber_1\!\Big(x_1-(x_2+\mu_2)\frac{a}{c}\Big)\Ber_1\!\Big((x_2+\mu_2)\frac{1}{c}\Big)
\Big)\\
&-\frac{\sgn(c(a+c\tau))}{4}\ind_\Z(x_1)\ind_\Z(x_2).
\end{align*}
Note that the addends concerning $\Ber_1$ do not depend on $\mu_1$, so their sum over $\mu\in\Z^2/\Z^2\sigma^\Tr$ is $2S(-a,c;x)$. Then we deduce $\ZZ_0=h\circ\ZZ$ from Lemma~\ref{B_2.sum}  by a direct computation.

Now we prove part (\ref{ZZ_0.thm.rati.}). It is clear that in this case we have $\gcd(a,c)=1$. On the other hand, conditions $\alpha^{-1}v\equiv v\mod\Lambda$ and $\gamma^{\Tr}\omega=\alpha\omega$ imply that $(x\gamma^{-\Tr})\omega\equiv x\omega\mod\Lambda$, from which we get $cx_1-ax_2\equiv -x_2\mod \Z$. Hence by the expression \eqref{ZZ_0.defi.} of $\ZZ_0(\gamma)(x,\omega)$ we obtain
\begin{align*}
\ZZ_0(\gamma)(x,\omega)=\frac{\alpha^{-1}}{2c}\Ber_2(-x_2)+\frac{\alpha}{2c}\Ber_2(x_2)+\sgn(c)S(-a,c;x)\\-\frac{\sgn(c\alpha)}{4}\ind_\Z(x_1)\ind_\Z(x_2).
\end{align*}
It is easy to see that $\Ber_2(-x_2)=\Ber_2(x_2)$, and that $\alpha^{-1}$ is the Galois conjugate of $\alpha$. Therefore the desired identity holds.

Finally, to deduce that $\ZZ_0$ represents a nonzero class in $H^1(G,\D_0)$, first note that conditions $\gamma^\Tr\omega=\alpha\omega$ and $\alpha^{-1}v\equiv v\mod\Lambda$ imply that $x\gamma^{-\Tr}\equiv x\mod\Z^2$. Then any 1-coboundary $\Xi$ for the action of $G$ on $\D_0$ gives $\Xi(\gamma)(x,\omega)=0$ for any $\gamma$ and $(x,\omega)$ as in Theorem~\ref{ZZ_0.thm}\eqref{ZZ_0.thm.rati.} with $\alpha>0$, as $$\nu(x\gamma^{-\Tr},\gamma^\Tr\omega)=\nu(x,\alpha\omega)=\nu(x,\omega)\qquad\text{for any $\nu\in\D_0$}.$$ 
Therefore, $\ZZ_0$ is not a 1-coboundary since $\ZZ_0(\gamma)(x,\omega)=-9/20$ for
$$\gamma=\begin{pmatrix}
26&-45\\-15&26
\end{pmatrix}, \qquad x=\Big(\frac{1}{5},0\Big), \qquad\omega=(5,5\sqrt{3})^\Tr, \qquad \alpha=26-15\sqrt{3}.$$
This completes the proof.
\end{proof}

\section{Proof of Theorem \ref{Hayes.Espinoza.thm.}}\label{Hayes.Espinoza.sect.}

Theorem~\ref{Hayes.Espinoza.thm.} will follow from \eqref{ZZ_0.defi.} by a rather direct computation together with a lattice point counting inspired by Hayes' proof. During this section we let 
\begin{align*}
\gamma=\begin{pmatrix}
a&b\\c&d
\end{pmatrix} \ \in \ \SL_2(\Z) \qquad\text{with $c>0$ and $\ell_\gamma>1$,}
\end{align*}
$x=(x_1,0)\in\Q^2$ with $x_1\in\ell_\gamma^{-1}\Z\smallsetminus\Z$, and $\omega\in\Mirr_{2\times1}(\R)$. We also set $\tau:=\omega_2/\omega_1$ and we recall the definition $\ell_\gamma:=\gcd(c,a-1)$ of the content of $\gamma$.

First note that \eqref{ZZ_0.defi.} gives
\begin{align*}
\ZZ_0(\gamma)(x,\omega)-&\ZZ_0(\gamma)(\ell_\gamma x,\ell_\gamma^{-1}\omega)=\\
&\frac{\sgn(a+c\tau)}{4}+\sum_{m\in\Z/c\Z}\Big(\Ber_1\!\Big(x_1-m\frac{a}{c}\Big)-\Ber_1\!\Big(-m\frac{a}{c}\Big)\Big)\Ber_1\!\Big(\frac{m}{c}\Big),
\end{align*}
as $cx_1\in\Z$. Then note that $x_1-ma/c$ is an integer if and only if $c$ divides $ma-cx_1$, so $c$ does not divide $m$ when $x_1-ma/c$ is an integer because $x_1\notin\Z$. Furthermore, $ad$ is congruent to $1$ modulo $c$ since $\gamma\in\SL_2(\Z)$. This shows that
\begin{align*}
\sum_{m\in\Z/c\Z}&\Big(\Ber_1\!\Big(x_1-m\frac{a}{c}\Big)-\Ber_1\!\Big(-m\frac{a}{c}\Big)\Big)\Ber_1\!\Big(\frac{m}{c}\Big)=\\
&\Ber_1(x_1)\Ber_1(dx_1)+\sum_{\substack{m\in\Z/c\Z\\ m\not\equiv cdx_1}}\Big(\Big\{x_1-m\frac{a}{c}\Big\}-\Big\{-m\frac{a}{c}\Big\}\Big)\Ber_1\Big(\frac{m}{c}\Big).
\end{align*}
Then if we let $A:=\big\{m\in\Z/c\Z \ : \ \{x_1\}+\{-ma/c\}\geq1\big\}$,
we get
\begin{align*}
\sum_{\substack{m\in\Z/c\Z\\ m\not\equiv cdx_1}}\Big(\Big\{x_1-m\frac{a}{c}\Big\}-&\Big\{-m\frac{a}{c}\Big\}\Big)\Ber_1\Big(\frac{m}{c}\Big)=\\
&\sum_{\substack{m\in\Z/c\Z\\m\not\equiv cdx_1}}\{x_1\}\Ber_1\Big(\frac{m}{c}\Big)-\sum_{\substack{m\in A\\ m\not\equiv cdx_1}}\Ber_1\Big(\frac{m}{c}\Big).
\end{align*}
Note that $\{-ma/c\}=\{-x_1\}$ if $c$ divides $m-cdx_1$, so $m\in A$ if $c$ divides $m-cdx_1$. On the whole we obtain
\begin{align}\label{Eq.2.1}
&\ZZ_0(\gamma)(x,\omega)-\ZZ_0(\gamma)(\ell_\gamma x,\ell_\gamma^{-1}\omega)\\
&\nonumber=\frac{\sgn(a+c\tau)}{4}+\Ber_1(x_1)\Ber_1(dx_1)+\sum_{\substack{m\in\Z/c\Z\\m\not\equiv cdx_1}}\{x_1\}\Ber_1\Big(\frac{m}{c}\Big)-\sum_{\substack{m\in A\\ m\not\equiv cdx_1}}\Ber_1\Big(\frac{m}{c}\Big)\\
&\nonumber=\frac{\sgn(a+c\tau)}{4}+\Ber_1(x_1)\Ber_1(dx_1)+\{x_1\}\sum_{\substack{m\in\Z/c\Z}}\Ber_1\Big(\frac{m}{c}\Big)-\sum_{\substack{m\in A}}\Ber_1\Big(\frac{m}{c}\Big)
\\ \nonumber & \qquad\qquad\qquad\qquad\qquad\qquad\qquad\qquad\qquad\qquad-\{x_1\}\Ber_1(dx_1)+\Ber_1(dx_1)\\
\nonumber &=\frac{\sgn(a+c\tau)}{4}+\Ber_1(dx_1)\big(\Ber(x_1)-\{x_1\}+1\big)-\sum_{\substack{m\in A}}\Ber_1\Big(\frac{m}{c}\Big)\\
\nonumber &=\frac{\sgn(a+c\tau)}{4}+\frac{1}{2}\Ber_1(dx_1)-\sum_{\substack{m\in A}}\Ber_1\Big(\frac{m}{c}\Big),
\end{align}
where we used that $\sum_{m\in\Z/c\Z}\Ber_1(m/c)=0$. 

\begin{lemma}\label{Lemma.2.1} If $A=\big\{m\in\Z/c\Z \ : \ \{x_1\}+\{-ma/c\}\geq1\big\}$, then the main term on the rightmost side of \eqref{Eq.2.1} satisfies
\begin{align*}
\sum_{\substack{m\in A}}\Ber_1\Big(\frac{m}{c}\Big)=-\frac{\{x_1\}}{2}c+\sum_{m=1}^{c\{x_1\}}\Big\{m\frac{d}{c}\Big\}.
\end{align*}
\end{lemma}

\begin{proof}
Elementary arguments show that 
\begin{align*}
\sum_{\substack{m\in A}}\Ber_1\Big(\frac{m}{c}\Big)=\sum_{\substack{m\in\Z/c\Z\\ \{x_1\}+\{m/c\}\geq1}}\Ber_1\Big(-m\frac{d}{c}\Big)=-\sum_{\substack{m\in\Z/c\Z\\ \{x_1\}+\{m/c\}\geq1}}\Ber_1\Big(m\frac{d}{c}\Big)\\
=-\sum_{\substack{m=c(1-\{x_1\})}}^{c-1}\Ber_1\Big(m\frac{d}{c}\Big).
\end{align*}
Note in the last sum that $c(1-\{x_1\})$ is an integer since $c\{x_1\}=c(x_1-\lfloor x_1\rfloor)$ is a positive integer, as $x_1\notin\Z$ and $cx_1\in\Z$. Then
\begin{align*}
-\sum_{\substack{m=c(1-\{x_1\})}}^{c-1}\Ber_1\Big(m\frac{d}{c}\Big)=\sum_{\substack{m=1-c}}^{c(\{x_1\}-1)}\Ber_1\Big(m\frac{d}{c}\Big)=\sum_{\substack{m=1}}^{c\{x_1\}}\Ber_1\Big(m\frac{d}{c}\Big)\\
=-\frac{1}{2}c\{x_1\}+\sum_{\substack{m=1}}^{c\{x_1\}}\Big\{m\frac{d}{c}\Big\},
\end{align*}
which proves the result.
\end{proof}

Equation~\eqref{Eq.2.1} and Lemma~\ref{Lemma.2.1} show that
\begin{align}\label{Eq.left.Hayes}
\ZZ_0(\gamma)(x,\omega)-&\ZZ_0(\gamma)(\ell_\gamma x,\ell_\gamma^{-1}\omega)=\\
&\nonumber\frac{\sgn(a+c\tau)}{4}+\frac{1}{2}\Ber_1(dx_1)+\frac{\{x_1\}}{2}c-\sum_{\substack{m=1}}^{c\{x_1\}}\Big\{m\frac{d}{c}\Big\}.
\end{align}
Then it only remains to show that the right-hand side of \eqref{Eq.left.Hayes} equals the right-hand side of \eqref{Hayes.Espinoza.formula}. For this purpose recall $T_{\gamma,x}$, the dilation $\{x_1\}T_\gamma$ of the triangle $T_\gamma$ in $\R^2$ having vertices at $(0,0)$, $(a-1,c)$, and $(-1,0)$. Let us first count the number of integral points belonging to the interior $T^\circ_{\gamma,x}$ of $T_{\gamma,x}$ in the Euclidean plane. 

Let $\{e_1, e_2\}$ be the canonical basis of $\R^2$ and let $\lambda=me_1+ne_2\in\Z^2$. Note that $\lambda$ lies in $T^\circ_{\gamma,x}$ if and only if $\lambda+\{x_1\}e_1$ lies in $T^\circ_{\gamma,x}+\{x_1\}e_1$. The vertices of the boundary of $\T:=T^\circ_{\gamma,x}+\{x_1\}e_1$ are the origin, $\{x_1\}e_1$, and $a\{x_1\}e_1+c\{x_1\}e_2$. So we look for $t_1, t_2\in\R$ such that $0<t_1<1$, $0<t_2<1-t_1$, and
\begin{align*}
\lambda+\{x_1\}e_1=t_1\{x_1\}e_1+t_2\big(a\{x_1\}e_1+c\{x_1\}e_2\big).
\end{align*} 
Consider the change of variables $m=ar+bt$ and $n=cr+dt$ with $r, t\in\Z$. This substitution does not have an effect on the generic quality of $\lambda$ since $\gamma\in\SL_2(\Z)$. Then 
\begin{align*}
\lambda+\{x_1\}e_1=\Big(\{x_1\}-\frac{t}{c}\Big)e_1+\Big(r+\frac{dt}{c}\Big)(ae_1+ce_2),
\end{align*}
so $\lambda+\{x_1\}e_1$ lies in $\T$ if and only if $0<t<c\{x_1\}$ and $(d-1)t/c<-r<dt/c$. Note that $dt/c$ is an integer if and only if $c$ divides $t$, but the latter is impossible since $0<t<c\{x_1\}<c$. Hence $\lambda+\{x_1\}e_1$ lies in $\T$ if and only if $0<t<c\{x_1\}$ and $(d-1)t/c<-r\leq\lfloor dt/c\rfloor$. Therefore 
\begin{align}\label{Numb.latt.point.}
\#(\Z^2\cap T^\circ_{\gamma,x})=\#\big((\Z^2+\{x_1\}e_1)\cap \T\big)=\sum_{m=1}^{c\{x_1\}-1}\Big\lfloor m\frac{d}{c}\Big\rfloor-\Big\lfloor m\frac{d-1}{c}\Big\rfloor.
\end{align}
Following Hayes we obtain the next lemma.

\begin{lemma}\label{Lemma.2.2}
Regarding Equation~\eqref{Numb.latt.point.}, we have
\begin{align*}
\#(\Z^2\cap T^\circ_{\gamma,x})=\frac{\{x_1\}}{2}\big(c\{x_1\}+1+c-\ell_\gamma\big)-\sum_{m=1}^{c\{x_1\}}\Big\{m\frac{d}{c}\Big\}.
\end{align*}
\end{lemma}

\begin{proof}
Elementary arguments show that
\begin{align*}
\sum_{m=1}^{c\{x_1\}-1}\Big\lfloor m\frac{d}{c}\Big\rfloor-\Big\lfloor m\frac{d-1}{c}\Big\rfloor=\sum_{m=1}^{c\{x_1\}-1}\Big(m\frac{d}{c}-\Big\{m\frac{d}{c}\Big\}\Big)-\sum_{m=1}^{c\{x_1\}-1}\Big\lfloor m\frac{d-1}{c}\Big\rfloor\\
=\big\{d\{x_1\}\big\}+\frac{d}{c}\sum_{m=0}^{c\{x_1\}-1}m-\sum_{m=1}^{c\{x_1\}}\Big\{m\frac{d}{c}\Big\}-\sum_{m=0}^{c\{x_1\}-1}\Big\lfloor m\frac{d-1}{c}\Big\rfloor.
\end{align*}
On the other hand, we note that
\begin{align}\label{content.change.prop}
\ell_\gamma:=\gcd(c,a-1)=\gcd(c,a+bc-ad)
=\gcd(c,a(1-d))=\gcd(c,d-1)
\end{align}
since $\gamma\in\SL_2(\Z)$, so $\ell_\gamma$ divides $d-1$. In particular, $(d-1)\{x_1\}$ is an integer, so we obtain $\big\{d\{x_1\}\big\}=\{x_1\}$. Also, if we let $d^*/c^*=(d-1)/c$, where $d^*:=(d-1)/\ell_\gamma$ and $c^*:=c/\ell_\gamma$ are coprime integers, we get
\begin{align*}
\sum_{m=0}^{c\{x_1\}-1}\Big\lfloor m\frac{d-1}{c}\Big\rfloor&=\sum_{m=0}^{c\{x_1\}-1}m\frac{d-1}{c}-\sum_{m=0}^{c\{x_1\}-1}\Big\{m\frac{d-1}{c}\Big\}\\
&=\frac{d}{c}\sum_{m=0}^{c\{x_1\}-1}m-\frac{1}{c}\sum_{m=0}^{c\{x_1\}-1}m-\sum_{j=0}^{\ell_\gamma\{x_1\}-1}\sum_{r=0}^{c^*-1}\Big\{(r+jc^*)\frac{d-1}{c}\Big\}\\
&=\frac{\{x_1\}}{2}\big(1-c\{x_1\}\big)+\frac{d}{c}\sum_{m=0}^{c\{x_1\}-1}m-\sum_{j=0}^{\ell_\gamma\{x_1\}-1}\sum_{r=0}^{c^*-1}\Big\{r\frac{d^*}{c^*}\Big\}\\
&=\frac{\{x_1\}}{2}\big(1-c\{x_1\}\big)+\frac{d}{c}\sum_{m=0}^{c\{x_1\}-1}m-\ell_\gamma\{x_1\}\sum_{r=0}^{c^*-1}\Big\{r\frac{d^*}{c^*}\Big\}\\
&=\frac{\{x_1\}}{2}\big(1-c\{x_1\}\big)+\frac{d}{c}\sum_{m=0}^{c\{x_1\}-1}m-\ell_\gamma\{x_1\}\sum_{r=0}^{c^*-1}\frac{r}{c^*}\\
&=\frac{\{x_1\}}{2}\big(1-c\{x_1\}\big)+\frac{d}{c}\sum_{m=0}^{c\{x_1\}-1}m-\frac{\{x_1\}}{2}\big(c-\ell_\gamma\big)\\
&=-\frac{\{x_1\}}{2}\big(c\{x_1\}-1+c-\ell_\gamma\big)+\frac{d}{c}\sum_{m=0}^{c\{x_1\}-1}m,
\end{align*}
where we used that ``multiplication by $d^*$'' induces an automorphism of the additive group $\Z/c^*\Z$. Therefore our result follows.
\end{proof}

Now we count the number of integral points $\lambda$ in the boundary $\partial T_{\gamma,x}$ of $T_{\gamma,x}$.  Again, we note that $\lambda$ lies in $\partial T_{\gamma,x}$ if and only if $\lambda+\{x_1\}e_1$  lies in $\partial T_{\gamma,x}+\{x_1\}e_1$. We start by considering the edge $\E_1$ of $\partial T_{\gamma,x}+\{x_1\}e_1$ joining the vertices $\{x_1\}e_1$ and $a\{x_1\}e_1+c\{x_1\}e_2$, both inclusive. Thus we look for $t_1, t_2\in\R$ such that $0\leq t_1\leq 1$, $t_2=1-t_1$, and
\begin{align*}
\lambda+\{x_1\}e_1=t_1\{x_1\}e_1+t_2\big(a\{x_1\}e_1+c\{x_1\}e_2\big).
\end{align*} 
Recall that for a generic $\lambda\in\Z^2$,
\begin{align*}
\lambda+\{x_1\}e_1=\Big(\{x_1\}-\frac{t}{c}\Big)e_1+\Big(r+\frac{dt}{c}\Big)(ae_1+ce_2)\qquad\qquad(r,t\in\Z),
\end{align*}
so we would like to find conditions on $r$ and $t$ such that $0\leq\{x_1\}-t/c\leq\{x_1\}$ and $r+dt/c=t/c$, which is equivalent to having $0\leq t\leq c\{x_1\}$ and $-r=d^*t/c^*$, where $d^*=(d-1)/\ell_\gamma$ and $c^*=c/\ell_\gamma$ as in the proof of Lemma~\ref{Lemma.2.2}. Then these conditions hold if and only if $c^*$ divides $t$, $0\leq t\leq c\{x_1\}$, and $-r=d^*t/c^*$. Therefore there are $1+\ell_\gamma\{x_1\}$ points of the form $\lambda+\{x_1\}e_1$ in $\E_1$. Hence we have proved the following.

\begin{lemma}\label{Lemma.2.3}
There are $1+\ell_\gamma\{x_1\}$ integral points in the edge of $T_{\gamma,x}$ joining the origin and $(a-1)\{x_1\}e_1+c\{x_1\}e_2$, both vertices inclusive.
\end{lemma}

Now we enumerate the number of points of the form $\lambda+\{x_1\}e_1$ in the edge $\E_2$ of $\partial T_{\gamma,x}+\{x_1\}e_1$ joining the origin (inclusive) and $\{x_1\}e_1$ (exclusive). So we look for $t_1\in\R$ such that $0\leq t_1< 1$ and
\begin{align*}
\lambda+\{x_1\}e_1=\Big(\{x_1\}-\frac{t}{c}\Big)e_1+\Big(r+\frac{dt}{c}\Big)(ae_1+ce_2)=t_1\{x_1\}e_1
\end{align*} 
for some $r, t\in\Z$, which would imply that $c$ divides $t$ and $0<t\leq c\{x_1\}$ for some $t\in\Z$, a contradiction. Then there are no points of the form $\lambda+\{x_1\}e_1$ in $\E_2$.

\begin{lemma}\label{Lemma.2.4}
There are no integral points in the edge of $T_{\gamma,x}$ joining the vertices $-\{x_1\}e_1$ (inclusive) and the origin (exclusive).
\end{lemma}

Finally we focus on points of the form $\lambda+\{x_1\}e_1$ in the line segment $\E_3$ joining the origin and $a\{x_1\}e_1+c\{x_2\}e_2$, both end points exclusive. In this case we seek $t_2\in\R$ such that $0<t_2<1$ and 
\begin{align*}
\lambda+\{x_1\}e_1=\Big(\{x_1\}-\frac{t}{c}\Big)e_1+\Big(r+\frac{dt}{c}\Big)(ae_1+ce_2)=t_2\big(a\{x_1\}e_1+c\{x_1\}e_2\big)
\end{align*}
for some integers $r$ and $t$. Note that the existence of such $t_2$ would imply the existence of $r\in\Z$ such that $(d-1)\{x_1\}<-r<d\{x_1\}$, which is impossible since $(d-1)\{x_1\}\in\Z$. Then there are no such $t_2$.

\begin{lemma}\label{Lemma.2.5}
There are no integral points in the edge of $T_{\gamma,x}$ joining  $-\{x_1\}e_1$ and $(a-1)\{x_1\}e_1+c\{x_1\}e_2$, both vertices exclusive.
\end{lemma}

Now we are in position to prove Theorem~\ref{Hayes.Espinoza.thm.}.

\begin{proof}[Proof of Theorem~\ref{Hayes.Espinoza.thm.}]
Combining Lemma~\ref{Lemma.2.2}, Lemma~\ref{Lemma.2.3}, Lemma~\ref{Lemma.2.4}, and Lemma~\ref{Lemma.2.5}, we obtain that
\begin{align*}
\Big(\sum_{\lambda\in\Z^2}w_{\gamma,x}(\lambda)\Big)&-\mathrm{area}(T_{\gamma,x})+\frac{\sgn(a+c\tau)-3}{4}\\
&=\frac{\{x_1\}}{2}\big(1+c\big)-\frac{1}{4}+\frac{\sgn(a+c\tau)}{4}-\sum_{m=1}^{c\{x_1\}}\Big\{m\frac{d}{c}\Big\}\\
&=\frac{1}{2}\Ber_1(x_1)+\frac{\{x_1\}}{2}c+\frac{\sgn(a+c\tau)}{4}-\sum_{m=1}^{c\{x_1\}}\Big\{m\frac{d}{c}\Big\}.
\end{align*}
Since $dx_1$ is congruent to $x_1$ modulo $\Z$, Equation~\eqref{Eq.left.Hayes} finishes the proof.
\end{proof}

\section{Proof of Theorem \ref{Ehrhart.Espinoza.thm}}\label{Ehrhart.Espinoza.sect.}
The proof of Theorem~\ref{Ehrhart.Espinoza.thm} can be summarized as follows. Firstly, we relate the right-hand side of \eqref{Hayes.Espinoza.formula} to the Ehrhart quasi-polynomials of certain triangles in $\R^2$. Then we establish a handy relation concerning the content $\ell_{\gamma\gamma'}$ of the product $\gamma\gamma'$ of a pair of $\SL_2(\Z)$ matrices $\gamma$ and $\gamma'$. Finally, we combine the previous ideas with the cocycle property of $\ZZ_0$ in order to get \eqref{Ehrhart.Espinoza.formula}. 

We start by briefly reviewing the results from the Ehrhart theory of polytopes that we will use subsequently. For a thorough treatment of this theory, we refer the reader to \cite{Li11}.

\subsection{Some elements of the Ehrhart theory of polytopes}
Let $\Po$ be a rational convex polytope in the $n$-dimensional Euclidean space $\R^n$, \ie $\Po$ is convex and it has vertices in $\Q^n$. Let $\G(\Po,m):=\#(\Z^n\cap m\Po)$ for all $m\in\Z_{\geq0}$. Let $\den(\Po)$ be the smallest positive integer $d$ such that $d\Po$ is an integral polytope, so $\den(\Po)\Po$ has vertices in $\Z^n$.

\begin{theorem*}[Ehrhart, 1962] If $\dim(\Po)$ denotes the dimension of $\Po$, then
$$\G(\Po,m)=\sum_{i=0}^{\dim(\Po)}\G_i(\Po,m)m^i\qquad\qquad\text{for all $m\in\Z_{\geq0}$,}$$
where each $\G_i(\Po):\Z_{\geq0}\longrightarrow \R$ satisfies $$\G_i\big(\Po,\,m+\den(\Po)\big)=\G_i(\Po,m)\qquad\qquad\text{for all $m\in\Z_{\geq0}$.}$$ 
\end{theorem*}

Thanks to the theorem above we say that $\G(\Po,m)$ is a quasi-polynomial in $m$ and that $\G_i(\Po)$ is the $i$th coefficient of $\G(\Po,m)$ for any $i\in\{0, \dots, \dim(\Po)\}$. In the case when $\Po$ has vertices in $\Z^n$, it verifies that $\G(\Po,m)$ is actually a polynomial in $m$, \ie all the $\G_i(\Po)$ are constants and they have furthermore a geometric meaning. For instance, if $n=\dim(\Po)=2$ we have
\begin{align}\label{Ehrhart.poly}
\G(\Po,m)=\mathrm{area}(\Po)m^2+\frac{1}{2}\#(\Z^2\cap \partial\Po)m+1 \qquad\qquad\text{for all $m\in\Z_{\geq0}$.}
\end{align}
In particular, Equation~\eqref{Ehrhart.poly} for $m=1$ is known as Pick's theorem.

For any $i\in\{0, \dots, \dim(\Po)\}$, denote by $g_i(\Po)$ the smallest positive integer $g$ such that $\G_i(\Po,m+g)=\G_i(\Po,m)$ for all $m\in\Z_{\geq0}$. 
Then we obviously have $g_i(\Po)\leq \den(\Po)$ for all $i\in\{0, \dots, \dim(\Po)\}$. However, we will need more refined information for which an additional concept of geometric nature is required. 

For any $i\in\{0, \dots, \dim(\Po)\}$, each face of $\Po$ of dimension $i$ is called an $i$-face of $\Po$. For any $M\subset\R^n$ consider the affine hull $\aff(M)$ of $M$, \ie the set
\begin{align*}
\aff(M):=\Bigg\{\sum_{j=1}^m\alpha_jx_j \ \Big| \ m\in\Z_{>0}, \ \alpha_j\in\R, \ \sum_{j=1}^m\alpha_j=1, \ x_j\in M\Bigg\}.
\end{align*}
Then we define the $i$-index $d_i(\Po)$ of $\Po$ as the smallest positive integer $d$ such that the intersection $\Z^n\cap \aff(dF)$ is nonempty for every $i$-face $F$ of $\Po$.

\begin{theorem*}[McMullen, 1978] The minimal period $g_i(\Po)$ of $\G_i(\Po)$ divides the $i$-index $d_i(\Po)$ of $\Po$, for any $i\in\{0, \dots, \dim(\Po)\}$.
\end{theorem*}

Finally, another well-known result states that the $\G_i(\Po)$ satisfy a certain homogeneity property. Indeed, if $m, t\in\Z_{\geq0}$, then
\begin{align}\label{Homo.prop.Gi}
\G_i(m\Po,t)=m^i\G_i(\Po,mt)\qquad\qquad\text{for all $i\in\{0, \dots, \dim(\Po)\}$.}
\end{align}
Moreover, we remark that Ehrhart's theorem shows that 
\begin{align}\label{Eq.G_0}
\G_0(\Po,0)=1.
\end{align}

\subsection{The $0$th coefficient of certain Ehrhart quasi-polynomials} In this subsection we consider a matrix
\begin{align*}
\gamma=\begin{pmatrix}
a&b\\c&d
\end{pmatrix} \ \in \ \SL_2(\Z) \qquad\text{with $c>0$}
\end{align*}
and content $\ell_\gamma=\gcd(c,a-1)$. Suppose that $\ell_\gamma>1$ and let $x=(m/\ell_\gamma,0)$ with $m\in\Z$ such that $0<m<\ell_\gamma$. Recall that $T_{\gamma}$ is the triangle in $\R^2$ with vertices at the origin, $(a-1,c)$, and $(-1,0)$, and also that $T_{\gamma,x}:=\{m/\ell_\gamma\}T_{\gamma}=(m/\ell_\gamma)T_{\gamma}$. Thanks to Lemma~\ref{Lemma.2.5}, Lemma~\ref{Lemma.2.4}, Lemma~\ref{Lemma.2.3}, and Ehrhart's theorem, we have
\begin{align*}
\G(\ell_\gamma^{-1}T_\gamma, m)&=\Big(\sum_{\lambda\in\Z^2}w_{\gamma,x}(\lambda)\Big)+\frac{1}{2}\#(\Z^2\cap \partial T_{\gamma,x})\\
&=\Big(\sum_{\lambda\in\Z^2}w_{\gamma,x}(\lambda)\Big)+\frac{1}{2}(1+m)\\
&=\G_2(\ell_\gamma^{-1}T_\gamma,m)m^2+\G_1(\ell_\gamma^{-1}T_\gamma,m)m+\G_0(\ell_\gamma^{-1}T_\gamma,m)
\end{align*}
with $w_{\gamma,x}$ as in Theorem~\ref{Hayes.Espinoza.thm.}. 

Note that $\ell_\gamma^{-1}T_\gamma$ has exactly two integral vertices: the origin and $\ell_\gamma^{-1}(a-1,c)$. This shows that $d_1(\ell_\gamma^{-1}T_\gamma)=d_2(\ell_\gamma^{-1}T_\gamma)=1$, whereas $d_0(\ell_\gamma^{-1}T_\gamma)=\ell_\gamma$. Hence McMullen's theorem implies that $\G_1(\ell_\gamma^{-1}T_\gamma,m)$ and $\G_2(\ell_\gamma^{-1}T_\gamma,m)$ are constant as functions of $m$. Then
\begin{align*}
\G_2(\ell_\gamma^{-1}T_\gamma,m)=\G_2(\ell_\gamma^{-1}T_\gamma,\ell_\gamma)=\G_2(T_\gamma,1)\Big(\frac{1}{\ell_\gamma}\Big)^2=\mathrm{area}(T_\gamma)\Big(\frac{1}{\ell_\gamma}\Big)^2=\frac{c}{2}\Big(\frac{1}{\ell_\gamma}\Big)^2,
\end{align*}
where we used \eqref{Homo.prop.Gi} and \eqref{Ehrhart.poly}. Similarly, 
\begin{align*}
\G_1(\ell_\gamma^{-1}T_\gamma,m)=\G_1(\ell_\gamma^{-1}T_\gamma,\ell_\gamma)=\G_1(T_\gamma,1)\Big(\frac{1}{\ell_\gamma}\Big)=\frac{1}{2}\#(\Z^2\cap \partial T_\gamma)\Big(\frac{1}{\ell_\gamma}\Big).
\end{align*}
Since $T_\gamma$ has only integral vertices, it is easy to prove that $\#(\Z^2\cap \partial T_\gamma)=\ell_\gamma+2$. All in all,
Theorem~\ref{Hayes.Espinoza.thm.} shows that
\begin{align}\label{Eq.4.2}
\ZZ_0(\gamma)(x,\omega)-\ZZ_0(\gamma)(\ell_\gamma x,\ell_\gamma^{-1}\omega)=\G_0(\ell_\gamma^{-1}T_\gamma,m)+\{x_1\}+\frac{\sgn(a+c\tau)-5}{4}
\end{align}
for all $\omega\in\Mirr_{2\times1}(\R)$, where $\{x_1\}=m/\ell_\gamma$ and $\tau=\omega_2/\omega_1$. 

We now show that Equation~\eqref{Eq.4.2} generalizes \eqref{Hayes.Espinoza.formula}. First note that the left-hand side of \eqref{Eq.4.2} is $0$ if $x_1$ is an integer, as $\ZZ_0(\gamma)(x,\omega)$ is $\Z^2$-periodic in the variable $x$ and homogeneous (of degree $0$) in the variable $\omega$. Then note that $\G_0(\ell_\gamma^{-1}T_\gamma,m)$ is $\ell_\gamma$-periodic in the variable $m$ since $d_0(\ell_\gamma^{-1}T_\gamma)=\ell_\gamma$. Furthermore, in view of \eqref{Eq.G_0}, the right-hand side of \eqref{Eq.4.2} equals $(\sgn(a+c\tau)-1)/4$ if we allow $m$ to be $0$ and $x_1$ to be an integer. Finally, if $\ell_\gamma=1$, the left-hand side of \eqref{Eq.4.2} is 0, whereas the right-hand side equals $(\sgn(a+c\tau)-1)/4$ again. Therefore, we have proved the following result.

\begin{proposition}\label{Ehrhart.Espinoza.Prop} 
Let $\gamma\in\SL_2(\Z)$ with $c>0$, let $m\in\Z_{\geq0}$, and let $\omega\in\Mirr_{2\times1}(\R)$. Then
\begin{align*}
\ZZ_0(\gamma)(x,\omega)-&\ZZ_0(\gamma)(\ell_\gamma x,\ell_\gamma^{-1}\omega)=\\
&\G_0(\ell_\gamma^{-1}T_\gamma,m)+\Big\{\frac{m}{\ell_\gamma}\Big\}-1+\Big(1-\ind_\Z\!\Big(\frac{m}{\ell_\gamma}\Big)\Big)\frac{\sgn(a+c\tau)-1}{4},
\end{align*}
where $x:=(m/\ell_\gamma,0)$ and $\tau:=\omega_2/\omega_1$.
\end{proposition}

\subsection{The content of the product of two matrices}

We now establish a handy relation concerning the content $\ell_{\gamma\gamma'}$ of the product of a pair of $\SL_2(\Z)$ matrices $\gamma$ and $\gamma'$. 

Let 
\begin{align}\label{gamma&gamma'}
\gamma=\begin{pmatrix}
a&b\\c&d
\end{pmatrix} \qquad\text{and}\qquad \gamma'=\begin{pmatrix}
a'&b'\\c'&d'
\end{pmatrix} \qquad \text{in $\SL_2(\Z)$.}
\end{align}
Note that the content of the product $\gamma\gamma'$ is 
\begin{align*}
\ell_{\gamma\gamma'}=\gcd(ca'+dc',\,a'a+bc'-1)
\end{align*}
by definition. Let $\ell:=\gcd(\ell_\gamma, \ell_{\gamma'})$. Since $\ell$ divides both $\ell_\gamma$ and $\ell_{\gamma'}$, $\ell$ divides $c$ and $c'$, so it also divides $ca'+dc'$. Similarly, we have that $\ell$ divides $a-1$, $a'-1$, and $c'$. Hence $\ell$ divides 
\begin{align*}
bc'+a(a'-1)+a-1=aa'+bc'-1.
\end{align*}
Therefore we have proved the following.

\begin{lemma}\label{content.mult.prop}
$\gcd(\ell_\gamma, \ell_{\gamma'})$ divides $\ell_{\gamma\gamma'}$.
\end{lemma}

\subsection{Proof of Theorem~\ref{Ehrhart.Espinoza.thm}}

We are now in position to address the proof of our last theorem. We let $\ell:=\gcd(\ell_\gamma,\ell_{\gamma'})$ for the matrices $\gamma$ and $\gamma'$ in \eqref{gamma&gamma'}. Suppose that $c$, $c'$, and $ca'+dc'$ are positive. We must prove \eqref{Ehrhart.Espinoza.formula}, \ie
\begin{align*}
\G_0(\ell^{-1}T_{\gamma\gamma'}\,,\,m)\,=\,\G_0(\ell^{-1}T_\gamma\,,\,m)\,+\,\G_0(\ell^{-1}T_{\gamma'}\,,\,m)\,+\,\{m/\ell\}\,-\,1
\end{align*}
for all $m\in\Z_{\geq0}$.

Let $m\in\Z_{\geq0}$ and $\omega\in\Mirr_{2\times1}(\R)$. We use the cocycle property of $\ZZ_0$ to show that
\begin{align*}
\ZZ_0(\gamma\gamma')\!\Big(\Big(\frac{m}{\ell},0\Big),\,\omega\Big)=\ZZ_0(\gamma)\!\Big(\Big(\frac{m}{\ell},0\Big),\,\omega\Big)+\ZZ_0(\gamma')\!\Big(\Big(\frac{m}{\ell},0\Big)\gamma^{-\Tr},\,\gamma^\Tr\omega\Big).
\end{align*}
Then we note that 
\begin{align*}
\Big(\frac{m}{\ell},0\Big)\gamma^{-\Tr}=\Big(\frac{md}{\ell}, -\frac{mc}{\ell}\Big)\equiv \Big(\frac{m}{\ell},0\Big)\mod\Z^2
\end{align*}
since $\ell$ divides both $d-1$ and $c$ (see \eqref{content.change.prop}). Hence,
\begin{align}\label{Eq.4.3}
\ZZ_0(\gamma\gamma')\!\Big(\Big(\frac{m}{\ell},0\Big),\,\omega\Big)=\ZZ_0(\gamma)\!\Big(\Big(\frac{m}{\ell},0\Big),\,\omega\Big)+\ZZ_0(\gamma')\!\Big(\Big(\frac{m}{\ell},0\Big),\,\gamma^\Tr\omega\Big).
\end{align}
Similarly, we get
\begin{align}\label{Eq.4.4}
\ZZ_0(\gamma\gamma')\!\big((m,0),\,\omega\big)=\ZZ_0(\gamma)\!\big((m,0),\,\omega\big)+\ZZ_0(\gamma')\!\big((m,0),\,\gamma^\Tr\omega\big)
\end{align}
by using that $(m,0)\gamma^{-\Tr}\equiv(m,0)\mod\Z^2$. Therefore, Proposition~\ref{Ehrhart.Espinoza.Prop}, \eqref{Eq.4.3}, \eqref{Eq.4.4}, and Lemma~\ref{content.mult.prop} imply that
\begin{align*}
&\G_0\!\Big(\frac{T_{\gamma\gamma'}}{\ell_{\gamma\gamma'}},\,\frac{\ell_{\gamma\gamma'}m}{\ell}\Big)+\Big\{\frac{m}{\ell}\Big\}-1+\Big(1-\ind_\Z\!\Big(\frac{m}{\ell}\Big)\Big)\frac{\sgn\big((a'a+bc')+(ca'+dc')\tau\big)-1}{4}\\
&=\G_0\!\Big(\frac{T_{\gamma}}{\ell_{\gamma}},\,\frac{\ell_{\gamma}m}{\ell}\Big)+\Big\{\frac{m}{\ell}\Big\}-1+\Big(1-\ind_\Z\!\Big(\frac{m}{\ell}\Big)\Big)\frac{\sgn(a+c\tau)-1}{4}\\
&\qquad\qquad+\G_0\!\Big(\frac{T_{\gamma'}}{\ell_{\gamma'}},\,\frac{\ell_{\gamma'}m}{\ell}\Big)+\Big\{\frac{m}{\ell}\Big\}-1+\Big(1-\ind_\Z\!\Big(\frac{m}{\ell}\Big)\Big)\frac{\sgn\big(a'+c'\frac{b+d\tau}{a+c\tau}\big)-1}{4}.
\end{align*}
Note that the last formula holds for all $\omega\in\Mirr_{2\times1}(\R)$. Then we choose $\tau$ big enough so that 
\begin{align*}
\tau>-\frac{aa'+bc'}{ca'+dc'}, \qquad\tau>-\frac{a}{c}, \qquad\tau>-\frac{c'b+a'a}{ca'+dc'}.
\end{align*}
This gives respectively
\begin{align*}
(a'a+bc')+(ca'+dc')\tau>0, \qquad a+c\tau>0, \qquad a'+c'\frac{b+d\tau}{a+c\tau}>0
\end{align*}
since $c$, $c'$, and $ca'+dc'$ are positive. Hence,
\begin{align*}
\G_0\!\Big(\frac{T_{\gamma\gamma'}}{\ell_{\gamma\gamma'}},\,\frac{\ell_{\gamma\gamma'}m}{\ell}\Big)
=\G_0\!\Big(\frac{T_{\gamma}}{\ell_{\gamma}},\,\frac{\ell_{\gamma}m}{\ell}\Big)
+\G_0\!\Big(\frac{T_{\gamma'}}{\ell_{\gamma'}},\,\frac{\ell_{\gamma'}m}{\ell}\Big)+\Big\{\frac{m}{\ell}\Big\}-1.
\end{align*}
Therefore, we finish the proof by using \eqref{Homo.prop.Gi}.


\begin{thebibliography}{XXX}

\bibitem[Eh62]{Eh62} Eug\`{e}ne Ehrhart, {\it{Sur les poly\`{e}dres rationnels homoth\'etiques \`{a} $n$ dimensions}}, Comptes Rendus des S\'eances de l'Acad\'emie des Sciences {\bf{S\'erie A, 254}} (1962), 616--618.

\bibitem[Es22]{Es22} Milton Espinoza, {\it{The Barnes-Hurwitz zeta cocycle on $\PGL_2(\Q)$}}, J. Number Theory {\bf{241}} (2022), 91--119.


\bibitem[Ha88]{Ha88} David R. Hayes, \textit{The Partial Zeta Functions of a Real Quadratic Number Field Evaluated at s = 0}, Number Theory: Proceedings of the First Conference of the Canadian Number Theory Association held at the Banff Center, Banff, Alberta, April 17–27, 1988, edited by Richard Mollin, Berlin, Boston: De Gruyter (2016), 207--226. 


\bibitem[Li11]{Li11} Eva Linke, {\it{Ehrhart polynomials, successive minima, and an Ehrhart theory for rational dilates of a rational polytope}}, Ph.D. thesis, Otto-von-Guericke-Universit\"at Magdeburg, 2011.


\bibitem[Mc78]{Mc78} Peter McMullen, {\it{Lattice invariant valuations on rational polytopes}}, Archiv der Mathematik {\bf{31-1}} (1978), 509--516.


\bibitem[Sh76]{Sh76} Takuro Shintani, {\it{On evaluation of zeta functions of totally real algebraic number fields at  non-positive integers}}, J. Fac. Sci. Univ.
Tokyo, Sec.\ IA {\bf{23}} (1976), 393--417.







\bibitem[Yo03]{Yo03}  Hiroyuki Yoshida, {\emph{Absolute CM-Periods}}, Mathematical Surveys and Monographs {\bf{106}}, American Mathematical Society, Providence, RI (2003).
\end{thebibliography}
\end{document}